\documentclass[12pt]{article}
\usepackage[a4paper, total={17cm,25cm}]{geometry}

\usepackage{amsmath,amssymb,amsfonts}
\usepackage{color, graphics}
\usepackage{tikz}
\usepackage{float}

\usepackage[alphabetic,y2k,initials]{amsrefs}

\usepackage[normalem]{ulem}

\def\arctan{\mathrm{arctan}}

\numberwithin{equation}{section}

\newtheorem{theorem}{Theorem}[section]
\newtheorem{proposition}[theorem]{Proposition}

\newtheorem{lemma}[theorem]{Lemma}
\newtheorem{remark}[theorem]{Remark}
\newtheorem{example}[theorem]{Example}

\newenvironment{proof}[1][Proof]{\noindent\textit{#1.} }{\hfill$\Box$\medskip}

 \title{Poncelet porism in singular cases}

\usepackage{authblk}
\author[1,3]{Vladimir Dragovi\'c}
\author[2,3]{Milena Radnovi\'c}
\affil[1]{\textsc{The University of Texas at Dallas, Department of Mathematical Sciences}}
\affil[2]{\textsc{The University of Sydney, School of Mathematics and Statistics}}
\affil[3]{\textsc{Mathematical Institute SANU, Belgrade}}
\affil[ ]{\texttt{vladimir.dragovic@utdallas.edu, milena.radnovic@sydney.edu.au}}

\date{}

\begin{document}

\maketitle

\begin{abstract}
The celebrated Poncelet porism is usually studied for a pair of smooth conics that are in a general position. Here we discuss Poncelet porism in the real plane -- affine or projective, when that is not the case, i.e.~the conics have at least one point of tangency or at least one of the conics is not smooth.
In all such cases, we find necessary and sufficient conditions for the existence of an $n$-gon inscribed in one of the conics and circumscribed about the other.
\end{abstract}

\begin{center}
 Dedicated to Academician Valery Vasilievich Kozlov on the occasion of his 75-th anniversary.
\end{center}

\tableofcontents

\section{Introduction}

For a given pair of conics in the plane, the famous Poncelet porism states: if there is a closed polygon inscribed in one of them and circumscribed about the other one, then there are infinitely many such polygons.
Usually, it is assumed that the conics are smooth and in a general position, meaning that they have only transversal intersections.
We note that the cases when conics are not smooth were considered by Poncelet \cite{Poncelet}, while the cases of non-transversal intersections were discussed  in the complex setting by Flatto \cite{FlattoBOOK} and they got that the poristic statements hold in those cases as well.

In this work, we analyze the cases where the two conics are in the real plane and they are either not in a general position or at least one if them is not smooth. 
We provide additional discussion to find necessary and sufficient conditions for the existence an $n$-gon inscribed in one of the conics and circumscribed about the other.

This paper is organised as follows.
In Section \ref{sec:non-transversal}, we consider the case of two smooth conics that have a point of tangency, while in Section \ref{sec:singular} we analyse the situation when one or both conics are not smooth.

This paper is dedicated to Academician V.~V.~Kozlov on the occasion of his 75-th anniversary. His book with Treschev \cite{KozTrBIL} is one of the classical references on mathematical billiards, which in the case of integrable billiards is closely related to the Poncelet porisms. Kozlov's paper \cite{Koz2003} is directly aimed toward the Poncelet porism. 

\section{Conics with non-transversal intersection}\label{sec:non-transversal}

In this section, we discuss polygonal lines inscribed in one smooth conic and circumscribed about another one, when those conics have at least one point of tangency.
We note that this setting, but in the complex plane, was analysed in \cite{FlattoBOOK} and the Poncelet-type porism jointly with the analytic conditions were derived, albeit without any examples.
An apparent contradiction between these results and the situation presented in  Figure \ref{fig:cusp}, motivated us to 
focus here  
to the specifics of the real case and provide another way to get the analytic conditions. We have not found in the literature any example that  visually presents a closed polygonal line inscribed in a conic and circumscribed about another conic, when the two conics have a point of tangency of order two. Thus, we provide the first such an example in Figure \ref{fig:double-triangle}. 	
	
\subsection{Conics in the real plane}\label{sec:conics-plane}

A \emph{smooth conic} or \emph{regular conic} in the real plane is a smooth curve given by a quadratic equation.
Such a curve divides the plane into two connected components:
\begin{itemize}
	\item one of those components consists of the \emph{inner points}: each line through such a point meets the conic at two distinct points;
	\item the other component consists of the \emph{outside points}: through any such point there are lines which are disjoint with the conic and there are also lines that intersect the conic at two points. Those two families of lines are separated by two \emph{tangent lines}.	
\end{itemize}
We note that all points on a tangent line to a conic are outside points, apart from the touching point.

Suppose that $\mathcal{C}$ and $\Gamma$ are two regular conics in the real plane.
Then those conics divide the plane into several connected components.
It is easy to see that each polygonal line inscribed in $\Gamma$ and circumscribed about $\mathcal{C}$ is contained in the closure of one of these components.
Moreover, that component is outside of $\mathcal{C}$.

\begin{proposition}\label{prop:touching-conics}
Let $\mathcal{C}$ and $\Gamma$ be regular conics in the real plane that are touching each other at point $T$.
Let $\mathcal{D}$ be one of the connected components into which the conics divide the plane, such that $T$ belongs to the boundary of $\mathcal{D}$.
Then we have:
\begin{itemize}
\item if $\mathcal{D}$ is inside $\mathcal{C}$, then polygonal lines which are inscribed in $\Gamma$, circumscribed about $\mathcal{C}$, and placed within $\mathcal{D}$  do not exist;
\item if $\mathcal{D}$ is outside $\mathcal{C}$, then there are no closed polygonal lines which are inscribed in $\Gamma$ and circumscribed about $\mathcal{C}$.
\end{itemize}
\end{proposition}
\begin{proof}
The first statement is obvious.

For the second one, let us first introduce a parametrization $\varphi:[0,1]\to\Gamma$ of the circumscribed conic $\Gamma$ such that $\varphi(0)=\varphi(1)=T$.
Suppose that $A_1A_2A_3\dots$ is a polygonal line inscribed in $\Gamma$ and circumscribed about $\mathcal{C}$.	 
If $\mathcal{C}$ is within $\Gamma$, as shown in Figure \ref{fig:cusp}, then the sequence $\left(\varphi^{-1}(A_1),\varphi^{-1}(A_2),\varphi^{-1}(A_3),\dots\right)$ is strictly monotone and bounded, thus the sequence of the vertices of the polygon is convergent.
 \begin{figure}[ht]
	\begin{center}
		\begin{tikzpicture}[scale=4]
			\draw[thick](0.1,0.3) circle (1);
			%		\draw[black,fill=gray](0.4,0) circle (0.02);
			%		\draw[black,fill=gray](-0.4,0) circle (0.02);
			\draw[thick,gray](0,0) ellipse (0.7695 and 0.657366);		
			\draw (-0.379426, 1.17758)--(1.07082, 0.0601866)--(0.342659, -0.670112)--(-0.210909, -0.65044)--(-0.427041, -0.54984)--(-0.497225, -0.502073)--(-0.519903, -0.484678);
			\draw[black,fill=black](-0.379426, 1.17758) circle (0.02);
			\draw[black,fill=black](1.07082, 0.0601866) circle (0.02);
			\draw[black,fill=black](0.342659, -0.670112) circle (0.02);
			\draw[black,fill=black](-0.210909, -0.65044) circle (0.02);
			\draw[black,fill=black](-0.427041, -0.54984) circle (0.02);
			\draw[black,fill=black](-0.497225, -0.502073) circle (0.02);
			\draw[black,fill=black](-0.519903, -0.484678) circle (0.02);
			
			\node[above] at (1.15,.6) { $\Gamma$};
			\node[above] at (-0.3,.4) { $\mathcal{C}$};
			
		\end{tikzpicture}
		\caption{Conic $\mathcal{C}$ is within $\Gamma$ and they have a point of tangency. The vertices of any polygonal line inscribed in $\Gamma$ and circumscribed about $\mathcal{C}$ will asymptotically approach the touching point of those two conics.}\label{fig:cusp}
	\end{center}
\end{figure}
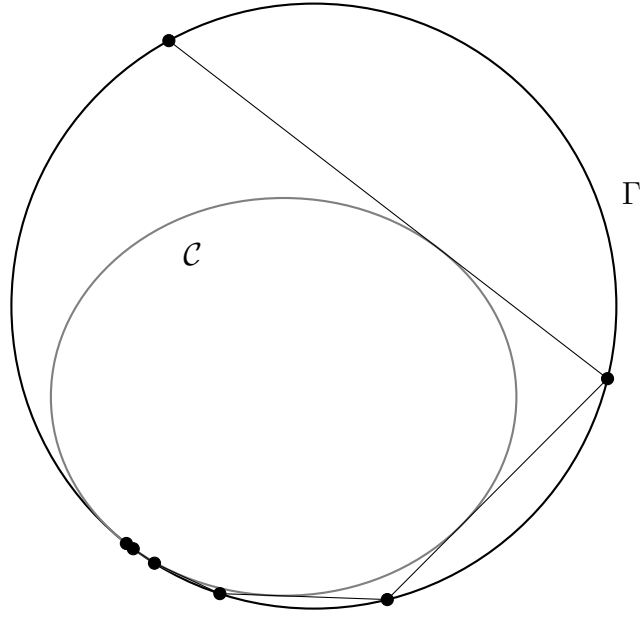

If $\mathcal{C}$ is outside $\Gamma$, as shown in Figure \ref{fig:cusp2}, we have that the subsequences $\left(\varphi^{-1}(A_1),\varphi^{-1}(A_3),\dots\right)$ and $\left(\varphi^{-1}(A_2),\varphi^{-1}(A_4),\dots\right)$ are monotone and bounded, so again the corresponding subsequences of the vertices will converge.
 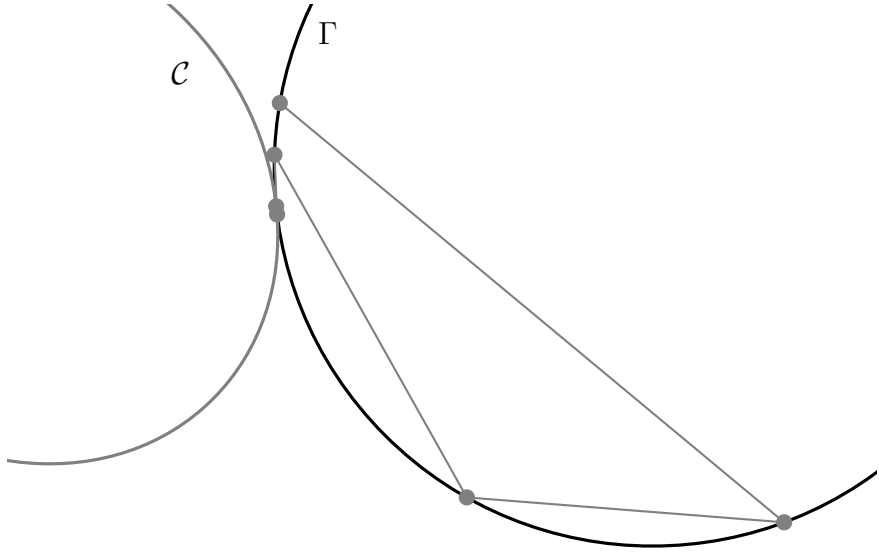
\begin{figure}[ht]
	\begin{center}
		\begin{tikzpicture}[scale=5]
			
			\clip (0,-1) rectangle (2.3,0.5);					
			
			\begin{scope}[rotate=-45]
				
				\draw[very thick](1.15921, 1.25382) circle (1);
				\draw[very thick,gray](0,0) ellipse (0.7695 and 0.657366);		
				\draw[gray,thick] (0.341298, 0.678479)--(2.06851, 0.83767)--(1.42839, 0.290725)--(0.427883, 0.571794)--(0.544391, 0.465151)--(0.527592, 0.47854);
				
				\draw[gray,fill=gray](0.341298, 0.678479) circle (0.02); 
				\draw[gray,fill=gray](2.06851, 0.83767) circle (0.02);
				\draw[gray,fill=gray](1.42839, 0.290725) circle (0.02); 
				\draw[gray,fill=gray](0.427883, 0.571794) circle (0.02); 
				\draw[gray,fill=gray](0.544391, 0.465151) circle (0.02); 
				\draw[gray,fill=gray](0.527592, 0.47854) circle (0.02);
				
				\node at (0.3,0.9) { $\Gamma$};
				\node at (0.1,.55) { $\mathcal{C}$};

			\end{scope}

		\end{tikzpicture}
		\caption{Conics $\Gamma$ and $\mathcal{C}$ are outside each other and they have a point of tangency. The vertices of any polygonal line inscribed in $\Gamma$ and circumscribed about $\mathcal{C}$ will asymptotically approach the touching point of those two conics.}\label{fig:cusp2}
	\end{center}
\end{figure}

Since the parameter of each vertex of the polygonal line depends continuously on the parameter of the first vertex, we get that the only possible limit in each of the cases is point $T$.
Thus, a polygonal line inscribed in $\Gamma$ and circumscribed about $\mathcal{C}$ cannot be closed.
\end{proof}

\begin{remark}
The polygonal line inscribed in $\Gamma$ and circumscribed about $\mathcal{C}$ in Figure \ref{fig:cusp} is reminiscent to yet undiscovered billiard trajectories having the limit at the cusp vertex.
We mention an interesting study \cite{King1} of billiards inside cusps. 
\end{remark}

\subsection{Double intersection point}\label{sec:1double}

In this section, we consider the case when conics $\Gamma$ and $\mathcal{C}$ have an intersection point of order $2$.

\begin{proposition}\label{prop:double-closed}
In the real plane, suppose that conics $\Gamma$ and $\mathcal{C}$ are given, such that they have a point of non-transversal intersection which is of order $2$.
Suppose that there is closed polygonal line inscribed in $\Gamma$ and circumscribed about $\mathcal{C}$.
Then, in addition to the point of order $2$, the conics also have two points of transversal intersection. 
Moreover, $\Gamma$ is touching $\mathcal{C}$ from within.
\end{proposition}

\begin{proof}
Denote by $T$ the double intersection point of $\Gamma$ and $\mathcal{C}$.
According to Proposition \ref{prop:touching-conics}, closed polygonal lines inscribed in $\Gamma$ and circumbscribed about $\mathcal{C}$ can exist in the real plane only if $\Gamma$ is touching $\mathcal{C}$ from within at $T$ and the two remaining intersection points of those conics are visible in the real plane, as presented in Figure \ref{fig:double-triangle}.
\end{proof}
 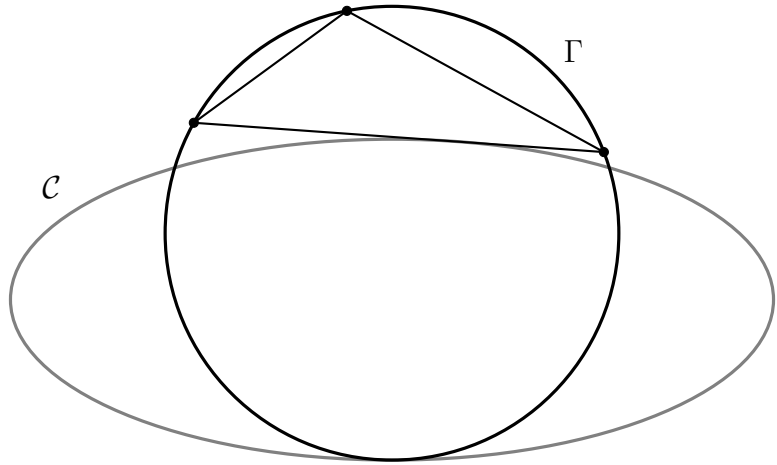
\begin{figure}[ht]
	\begin{center}
		\begin{tikzpicture}[scale=3]
			\draw[very thick,gray](0,0) ellipse (1.68179 and 0.707107);

			\draw[very thick](0,0.292893) circle (1);

			\draw[thick](-0.198669, 1.27296)--(0.933866, 0.650516)--(-0.873529, 0.779665)--cycle;
			\draw[black,fill=black](-0.198669, 1.27296) circle (0.02);
			\draw[black,fill=black](0.933866, 0.650516) circle (0.02);
			\draw[black,fill=black](-0.873529, 0.779665) circle (0.02);
			
			\node[above] at (0.8,1) {$\Gamma$};
			\node[above] at (-1.5,.4) {$\mathcal{C}$};
			
		\end{tikzpicture}
		\caption{There are triangles inscribed in $\Gamma$ and circumscribed about $\mathcal{C}$.}\label{fig:double-triangle}
	\end{center}
\end{figure}

\begin{example}\label{ex:triangle}
A pair of conics with a touching point with a triangle inscribed in one of them and circumscribed about the other one is shown in Figure \ref{fig:double-triangle}.
According to the Poncelet porism, there is an infinite family of such triangles. 
\end{example}
\begin{example}\label{ex:4isoperiodic}
It is proved in \cite{DR2024geom} that there are quadrangles inscribed in a circle and circumscribed about any conic from a confocal family, whenever the foci lie on the circle. Note that in such a confocal family, there will be a unique ellipse touching the circle. 
\end{example}	

We use a convenient coordinate system for further discussion of the closure conditions.

\begin{lemma}\label{lem:parabola}
Let $\Gamma$ and $\mathcal{C}$ be conics in the real projective plane, such that they have a joint point of order $2$.
Then one can choose an affine chart and a coordinate system such that the conics have the following equations:
\begin{equation}\label{eq:konike2}
\mathcal{C}:\ y=x^2
\quad\text{and}\quad
\Gamma:\ 
\alpha y= x^2+\beta xy+\gamma y^2,
\quad\text{with}\quad \alpha\neq1.
\end{equation}
\end{lemma}
\begin{proof}
We choose the coordinate system in the {projective} plane as it was done in \cite{FlattoBOOK}: such that the point of joint tangency is placed at the origin {of the affine chart $(x,y)$} and $\mathcal{C}$ is the parabola $y=x^2$.
With such a choice, the conic $\Gamma$ will have the equation as stated in \eqref{eq:konike2}. 
We note that $\alpha\neq1$, otherwise the order of tangency would be higher than $2$.
\end{proof}

Next, we derive the condition for the closure of the polygon in the real plane.

\begin{theorem}\label{th:cayley}
In the real plane, consider the conics $\mathcal{C}$ and $\Gamma$ given by \eqref{eq:konike2}.
Then there is a closed polygonal line with $n$ sides inscribed in $\Gamma$ and circumscribed about $\mathcal{C}$ if and only if the following conditions are satisfied:
\begin{itemize}
	\item[(a)] $\alpha=\cos^2\dfrac{\pi m}n$, where $m$, $n$ are positive integers with $2m<n$; and
	\item[(b)] $\beta^2-4\gamma(1-\alpha)>0$.
\end{itemize}
\end{theorem}
\begin{proof}
First, according to Proposition \ref{prop:double-closed}, conics have two simple intersection points and one double point, therefore, from \eqref{eq:konike2}, we get condition (b).

In order to obtain (a), we will consider the Cayley type conditions for closed polygons inscribed in one conic and circumscribed about the other one. See also \cite{FlattoBOOK}, for another approach to proving this.

The matrices for $\mathcal{C}$ and $\Gamma$ are:
$$
\mathcal{C}=
\begin{pmatrix}
-1 & 0 & 0\\
0 & 0 & 1/2\\
0 & 1/2 & 0
\end{pmatrix},
\qquad
\Gamma=
\begin{pmatrix}
	-1 & -\beta/2 & 0\\
	-\beta/2 & -\gamma & \alpha/2\\
	0 & \alpha/2 & 0
\end{pmatrix}.
$$
We have:
$$
\det(\mathcal{C}+\lambda\Gamma)=\frac14(1+\lambda)(1+\alpha\lambda)^2,
$$
thus the underlying algebraic curve is:
\begin{equation}\label{eq:singular-curve}
\mu^2=\frac14(1+\lambda)(1+\alpha\lambda)^2,
\end{equation}
which has an ordinary double point at $(\lambda,\mu)=(-1/\alpha,0)$, thus its normalization is the following rational curve:
\begin{equation}\label{eq:normalization}
\mu_1^2=1+\lambda_1,
\end{equation}
with the projection $\lambda=\lambda_1$, $\mu=\frac12\mu_1(1+\alpha\lambda_1)$.

The condition for existence of an $n$-polygon inscribed in $\Gamma$ and circumscribed about $\mathcal{C}$ is that there is a function on the normalized curve \eqref{eq:normalization} with zero of order $n$ at $(\lambda_1,\mu_1)=(0,1)$, pole of order $n$ at $(\lambda_1,\mu_1)=(\infty,\infty)$, no other zeros or poles, and having the same values at the two points:
$$
(\lambda_1,\mu_1)=\left(-\frac1{\alpha},\sqrt{1-\frac{1}{\alpha}}\right)
\quad\text{and}\quad
(\lambda_1,\mu_1)=\left(-\frac1{\alpha},-\sqrt{1-\frac{1}{\alpha}}\right),
$$
which are projected onto the double point of the first curve \eqref{eq:singular-curve}.
Such a function is of the form $(\mu_1-1)^n$, with the additional condition:
$$
\left(\sqrt{1-\frac1{\alpha}}-1\right)^n
=
\left(-\sqrt{1-\frac1{\alpha}}-1\right)^n.
$$
The last condition means that
$$
\left(1-\sqrt{1-\frac1{\alpha}}\right)\left(1+\sqrt{1-\frac1{\alpha}}\right)^{-1}
$$
is an $n$-th root of unity, which is equivalent to condition (a).
\end{proof}

\begin{remark}
Notice that since $\alpha\in(0,1)$, the conic $\Gamma$ will touch $\mathcal{C}$ from within, which is well in accordance with Proposition \ref{prop:double-closed}.
\end{remark}

\begin{remark}
If condition (a) of Proposition \ref{prop:double-closed} is satisfied, but not condition (b), ie, $\beta^2-4\gamma(1-\alpha)\le0$, then we will have that $\Gamma$ is within $\mathcal{C}$ and the conics may have one or two double points.
According to Proposition \ref{prop:touching-conics}, the real plane will contain no polygonal lines inscribed in $\Gamma$ and circumscribed about $\mathcal{C}$.
However, in the complexified plane, such polygonal lines  exist and they will be closed. This is the result obtained in \cite{FlattoBOOK}, which showed that  condition (a) is sufficient and necessary for the existence of an $n$-gon inscribed in $\Gamma$ and circumscribed about $\mathcal{C}$ in the complex case. 
\end{remark}

\begin{remark}
The condition (a) from Proposition \ref{prop:double-closed} can be expressed also in the form of the classical Cayley conditions \cites{Cayley1853,LebCONIQUES}, since push-downs to \eqref{eq:singular-curve} of such functions can be expressed as linear combinations of:
$$
1,\lambda,\dots,\lambda^n, \mu, \lambda\mu, \dots, \lambda^{n-2}\mu.
$$
\end{remark}

\begin{remark}
The condition (a) from Proposition \ref{prop:double-closed} can be derived also from the classical Chebyshev polynomials. A relationship between the generalized Cayley conditions for periodic billiard trajectories within ellipsoids in the $d$-dimensional Euclidean space and generalized Chebyshev polynomials on $d$ real intervals was discovered in \cite{DR2019cmp}. The case of two smooth conics in general position in the Euclidean plane, $d=2$, was considered in \cite{DR2019rcd}, where the relationship with generalized Chebyshev polynomials on two intervals (aka Zolotarev polynomials) and the classical Cayley condition was studied in detail.
\end{remark}

Let us recall (see e.g.~\cite{Akh4}) that the Chebyshev polynomials $T_n(x)$, $n\in\{0, 1, 2,\dots\}$,
are defined by the recursion:
\begin{equation}\label{eq:cheb1}
T_0(x)=1, \quad T_1(x)=x,\quad T_{n+1}(x)+T_{n-1}(x)=2xT_n(x).
\end{equation}
 They can be parameterized as
\begin{equation}\label{eq:cheb2}
T_n(x)=\cos n\phi,\quad x=\cos\phi.
\end{equation}

Denote $L_0=1$ and $L_n=2^{1-n}$, $n\in\{1, 2,\dots\}$.
Then the Chebyshev Theorem states that the polynomials $L_nT_n(x)$ are characterized as the solutions of the following minmax problem:
\begin{quote}
\emph{find the polynomial of degree $n$ with the leading coefficient equal 1 which minimizes the uniform norm on the interval $[-1, 1]$.}	
\end{quote}

For the Chebyshev polynomial $T_n$ there exists the Chebyshev polynomial of the second kind $Q_{n-1}$ of degree $n-1$ such that \emph{the polynomial Pell's equation} is satisfied:
$$
T_n^2(x)-(x-1)(x+1)Q_{n-1}^2(x)=1.
$$
The zeros of the polynomial $Q_{n-1}$ all lie in the open interval $(-1, 1)$. They are the internal extremal points of the Chebyshev polynomial $T_n$, given by $x_k=\cos(k\pi/n)$, for $k\in\{1, \dots, n-1\}$.

\begin{remark}
The condition of the existence of $n$-gons associated with the singular curve \eqref{eq:normalization} can be obtained from the polynomial Pell's equation for the Zolotarev polynomial of degree $n$ on  two intervals $[-1, c_3]\cup [c_2,0]$ when the two intervals merge into one interval $[-1, 0]$, while $c_3, c_2$ tend to $-\alpha$.

Thus, the Zolotarev polynomial of degree $n$ tends to a polynomial $R_n$, an extremal polynomial of degree $n$ on the interval $[-1, 0]$. It satisfies  polynomial Pell's equation
\begin{equation}\label{eq:Pell}
	R_n^2(x) - x(x+1)S^2_{n-1}(x)=1,
\end{equation}
where $S_{n-1}$ is a polynomial of degree $n-1$, which satisfies the condition $S_{n-1}(-\alpha)=0$.

Thus, $R_n$ is the Chebyshev polynomial $T_n$ composed with affine  transformation $\delta:[-1, 0]\rightarrow [-1, 1]$. It is obvious that $\delta(x)=2x+1$. The point $x=-\alpha$ corresponds to an internal extremal point of $R_n=T_n\circ \delta$.

We conclude that $\alpha$ has to satisfy
$$
\delta(-\alpha) = x_k= \cos \left(\frac{k\pi}{n}\right).
$$
Using well-known trigonometric identities,  from the last relation we derive condition (a) from Proposition \ref{prop:double-closed}, as indicated.
\end{remark}

We conclude this section by mentioning that a similar consideration was applied in \cite{ADR2019b}, in the description of periodic light-like billiard trajectories within ellipses in the Lorentzian plane.

\subsection{Triple and quadruple point}\label{sec:triple}

In this section, we discuss the cases when the conics have a joint point of order higher than $2$.

\begin{proposition}\label{prop:triple}
If conics $\Gamma$ and $\mathcal{C}$ have joint point of order higher than $2$, then there are no closed polygonal lines inscribed in $\Gamma$ and circumscribed about $\mathcal{C}$.
\end{proposition}

\begin{proof}
If conics $\mathcal{C}$ and $\Gamma$ have a triple joint point, then they will also have another point of a simple intersection, as shown in Figure \ref{fig:triple}.
 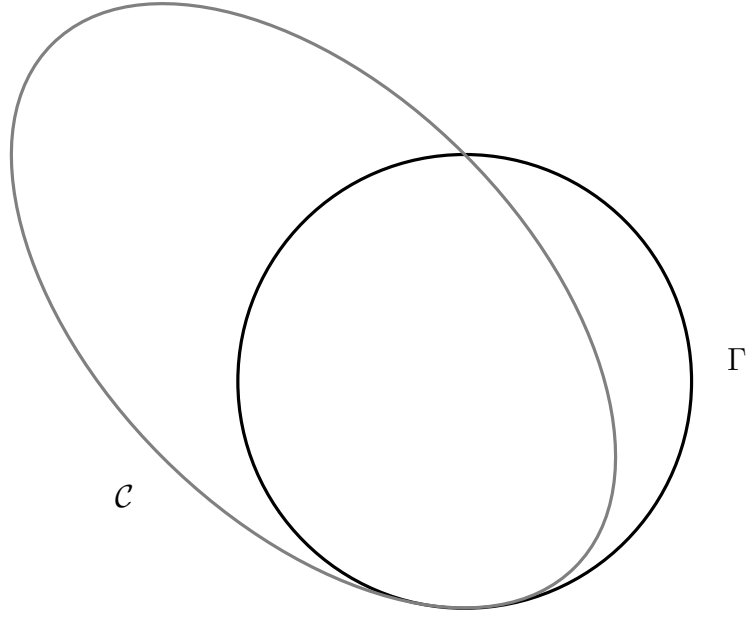
\begin{figure}[ht]
	\begin{center}
		\begin{tikzpicture}[scale=3]
			\draw[very thick](0,1) circle (1);
			
			%\draw[very thick,gray](0,0) ellipse (1.68179 and 0.707107);		

\draw[very thick,gray,rotate around={-45:(-0.666666,1.33333)}] (-0.666666,1.33333) ellipse (1.6299 and 0.942809);
			
			\node[above] at (1.2,1) {$\Gamma$};
			\node[above] at (-1.5,.4) {$\mathcal{C}$};
			
		\end{tikzpicture}
		\caption{Conics $\Gamma$ and $\mathcal{C}$ have a triple point of intersection and a simple one.}\label{fig:triple}
	\end{center}
\end{figure}
Such conics divide the plane into four connected components, and each one has the point of tangency at the boundary.
Thus, by Proposition \ref{prop:touching-conics}, there will be no closed polygons inscribed in one of them and circumscribed about the other one -- all polygons will have an infinite number of vertices.

If the joint point is of order $4$, then one of the conics is within the other one and the touching point is on the boundary of each connected component of the plane, so we can again conclude that there will be no closed polygonal lines.
\end{proof}

As in Lemma \ref{lem:parabola}, we can choose appropriate coordinate system, such that the touching point is at the origin, and the conics are given by:
$$
\mathcal{C}:\ y=x^2
\quad\text{and}\quad
\Gamma:\ 
y= x^2+\beta xy+\gamma y^2,
$$
with $\beta\neq0$ for triple point and $\beta=0$ for quadruple point.

The matrices for $\mathcal{C}$ and $\Gamma$ are:
$$
\mathcal{C}=
\begin{pmatrix}
	-1 & 0 & 0\\
	0 & 0 & 1/2\\
	0 & 1/2 & 0
\end{pmatrix},
\qquad
\Gamma=
\begin{pmatrix}
	-1 & -\beta/2 & 0\\
	-\beta/2 & -\gamma & 1/2\\
	0 & 1/2 & 0
\end{pmatrix}.
$$
We have:
\begin{equation}\label{eq:triple}
\det(\mathcal{C}+\lambda\Gamma)=\frac14(1+\lambda)^3.
\end{equation}

\begin{remark}
Let us look into Pell's equation \eqref{eq:Pell} and why it does not have any solutions for $\alpha=1$. 
The value $\alpha = 1$ corresponds to the case of \eqref{eq:triple} above, from the triple or quadruple point of tangency of two conics.

For $\alpha =1$, we get that $R_n'(-1)=0$. Thus, at $x=-1$, polynomial $R_n$ has a local extremal point. Thus, there exists a $\xi<-1$, arbitrarily close to $-1$, such that $R_n^2(\xi)<1$. Since $\xi(\xi+1)S_{n-1}^2(\xi)>0$, we get that 
$$
R_n^2(\xi) - \xi(\xi+1)S^2_{n-1}(\xi)<1,
$$
which contradicts Pell's equation \eqref{eq:Pell}.
\end{remark}

\section{Singular conics}\label{sec:singular}

In this section, we will consider \emph{singular} or \emph{degenerate} conics in the real plane.
For us, the relevant case of such a conic is when it is given by a quadratic which can be factorised into distinct linear components, i.e.~the singular conic is the union of two distinct real lines.
It is straightforward to construct polygonal lines inscribed in such degenerate conics, but all circumscribed polygons become trivial.

In order to be able to discuss the case when the inscribed conic is not regular, let us observe that the tangent lines to any regular conic are points on the dual of that conic in the dual plane.
A degenerate conic in the dual plane consists of the lines from two pencils of lines in the original plane, i.e.~of all the lines containing one of the two given points in the original plane.
Now, such a degenerate conic allows construction of circumscribed polygons: the sides of the polygon will alternately contain the points of a given pair in the plane.

\subsection{Inscribed conic is singular}\label{sec:in-singular}

Suppose that $\Gamma$ is a smooth conic in the real plane, while $\mathcal C^*$ is a singular conic in the dual plane, consisting of two pencils of lines.
Each of the two pencils consists of the lines through a given point, thus $\mathcal{C}^*$ is given by a pair of distinct points $C_1$ and $C_2$.
An example of polygonal lines inscribed in a smooth conic and circumscribed about a singular one is shown in Figure \ref{fig:C-singular}.
 \begin{figure}[ht]
	\begin{center}
		\begin{tikzpicture}[scale=3]
			\draw[very thick,gray](0,0) ellipse (1.41421 and 1);		
\draw[very thick,black] (1.40715, -0.0998334)--(0.107671, -0.997098)--(0.396114, 0.959972)--(-0.45474, 0.946893)--(-1.24894, -0.469122)--(1.29056, 0.408939)--cycle;

\draw[very thick,dashed,black](-1.40006, 0.14112)--(1.09941, 0.629009)--(1.27874, -0.42709)--(0.798247, -0.82547)--(0.753913, 0.846054)--(-0.969672, 0.72792)--cycle;
			
\draw[gray] (1.40715, -0.0998334)--(3,1)--(1.29056, 0.408939)--(0.696784,3)--(0.396114, 0.959972)--(3,1);

\draw[gray](0.696784,3)--(-0.45474, 0.946893);

\draw[gray,dashed] (3,1)--(0.753913, 0.846054)--(0.696784,3)--(1.09941, 0.629009)--(3,1)--(1.27874, -0.42709);

\draw[gray,dashed] (0.696784,3)--(-0.969672, 0.72792);

\draw[black,fill=black](1.40715, -0.0998334) circle (0.03);
\draw[black,fill=black](0.107671, -0.997098) circle (0.03);
\draw[black,fill=black](0.396114, 0.959972) circle (0.03);
\draw[black,fill=black](-0.45474, 0.946893) circle (0.03);
\draw[black,fill=black](-1.24894, -0.469122) circle (0.03);
\draw[black,fill=black](1.29056, 0.408939) circle (0.03);

\draw[black,fill=black](-1.40006, 0.14112) circle (0.03); 
\draw[black,fill=black](1.09941, 0.629009) circle (0.03); 
\draw[black,fill=black](1.27874, -0.42709) circle (0.03); 
\draw[black,fill=black](0.798247, -0.82547) circle (0.03); 
\draw[black,fill=black](0.753913, 0.846054) circle (0.03); 
\draw[black,fill=black](-0.969672, 0.72792) circle (0.03);

			\draw[gray,fill=gray](3, 1) circle (0.03) node[black,right]{$C_2$};
			\draw[gray,fill=gray](0.696784,3) circle (0.03) node[black,right]{$C_1$};

		\node[below left] at (-1,-.7) { $\Gamma$};
			
		\end{tikzpicture}
		\caption{Hexagons inscribed in a smooth conic $\Gamma$ and circumscribed about a singular conic $\mathcal{C}^*$, which is given by a pair of points $C_1$, $C_2$}.\label{fig:C-singular}
	\end{center}
\end{figure}
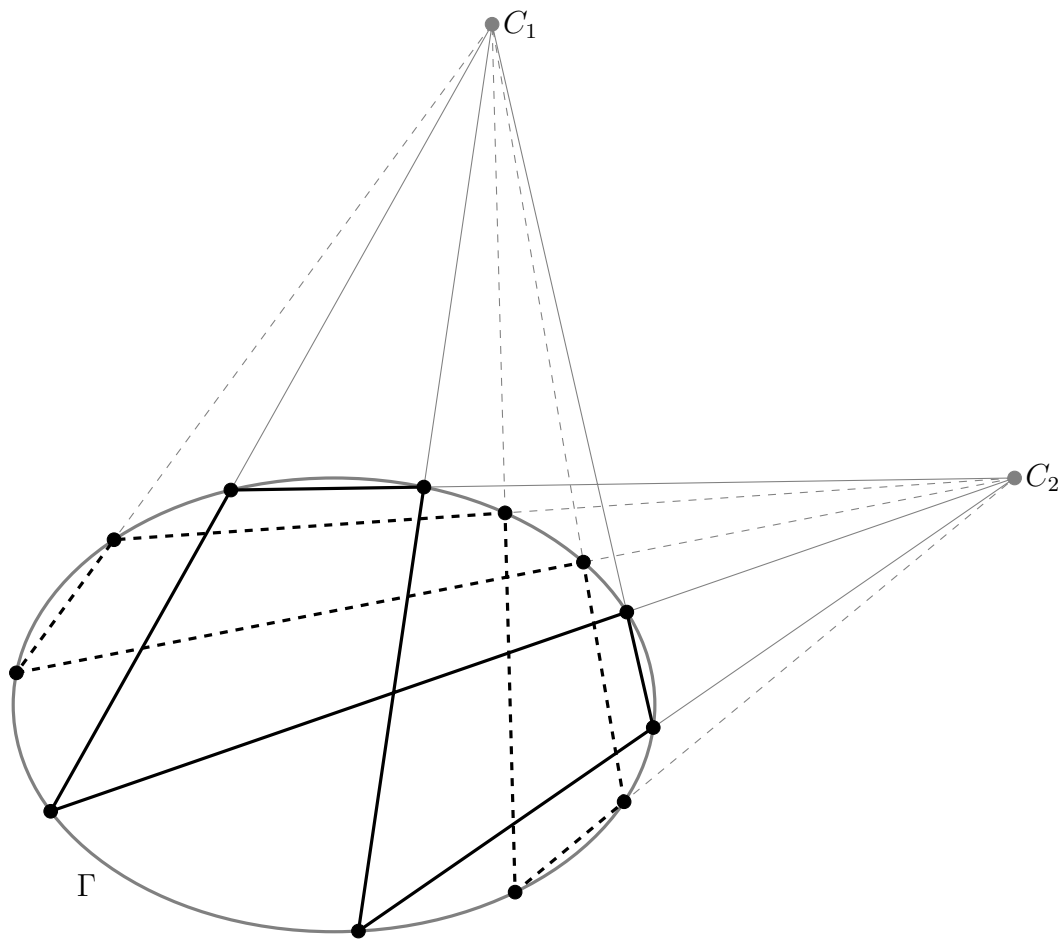

\begin{example}\label{ex:focal}
Any elliptic billiard trajectories whose first segment contains a focus of the boundary will represent a polygonal line inscribed in a smooth conic circumscribed about the singular conic determined by the foci of the billiard boundary.
For more details, see for example \cites{KozTrBIL,DR2011knjiga}.

We note that there are two types of polygonal lines in this case:
\begin{itemize}
	\item the line consisting of only one segment, containing both foci;
	\item the segments of all other polygonal lines alternatively contain the two foci and, moreover, they asymptoticaly approach the segment through the foci.
\end{itemize}

\end{example}

\begin{example}\label{ex:focal-lorentz}
We can also consider elliptic billiard in the Lorentzian plane, see \cites{BM1962,DR2012adv}. 
Such billiards also have focal property, but the foci of an ellipse are placed outside that curve.
There are periodic billiard trajectories that are placed on the line containing a pair of foci, while otherwise, they will asymptotically approach that line, similarly as in Example \ref{ex:focal}.
\end{example}

\begin{example}\label{ex:focal-gen}
Similar properties as in Examples \ref{ex:focal} and \ref{ex:focal-lorentz} will be shared whenever the line $C_1C_2$ has non-empty intersection with $\Gamma$.
One polygonal line will consist of a single segment on $C_1C_2$, while the segments of any other polygonal line will asymptotically approach that segment, see Figure \ref{fig:asymp}.
\end{example}

Examples \ref{ex:focal}--\ref{ex:focal-gen} can be generalised as follows:

\begin{proposition}\label{prop:asymp}
	Let $\Gamma$ be a smooth conic in the real plane, and $\mathcal{C}^*$ a singular one in the dual plane, given by a pair of points $C_1$, $C_2$ such that the line $C_1C_2$ has a non-empty intersection with $\Gamma$.
	
	Then all the polygonal lines inscribed in $\Gamma$ and circumscribed about $\mathcal{C}$ satisfy the following:
\begin{itemize}
	\item there is one such polygonal line consisting of a single segment $VW$, where points $V$, $W$ are the intersection of the line $C_1C_2$ and $\Gamma$;
	\item all other polygonal lines have infinitely many segments, and those segments asymptotically approach $VW$.
\end{itemize}
\end{proposition}
\begin{proof}
The first statement is obvious.
For the second statement, suppose first that $C_1$ is within $\Gamma$, while $C_2$ is outside, as shown in Figure \ref{fig:asymp}.
Since every second segment of a polygonal line $A_1A_2A_3\dots A_n\dots$ inscribed in $\Gamma$ and circumscribed about $\mathcal{C}^*$ intersects the segment $VW$, and that segment divides $\Gamma$ into two arcs, we will have that vertices $A_1$, $A_2$, $A_5$, $A_6$, \dots belong to one of those arcs, while $A_3$, $A_4$, $A_7$, $A_8$, \dots are on the other arc.
Moreover, the sequences $A_1$, $A_3$, $A_5$, \dots and $A_2$, $A_4$, $A_6$, \dots converge to $V$ and $W$ respectively.
Similar considerations can be done for the cases when both points $C_1$, $C_2$ are within $\Gamma$ or both are outside.

\end{proof}
 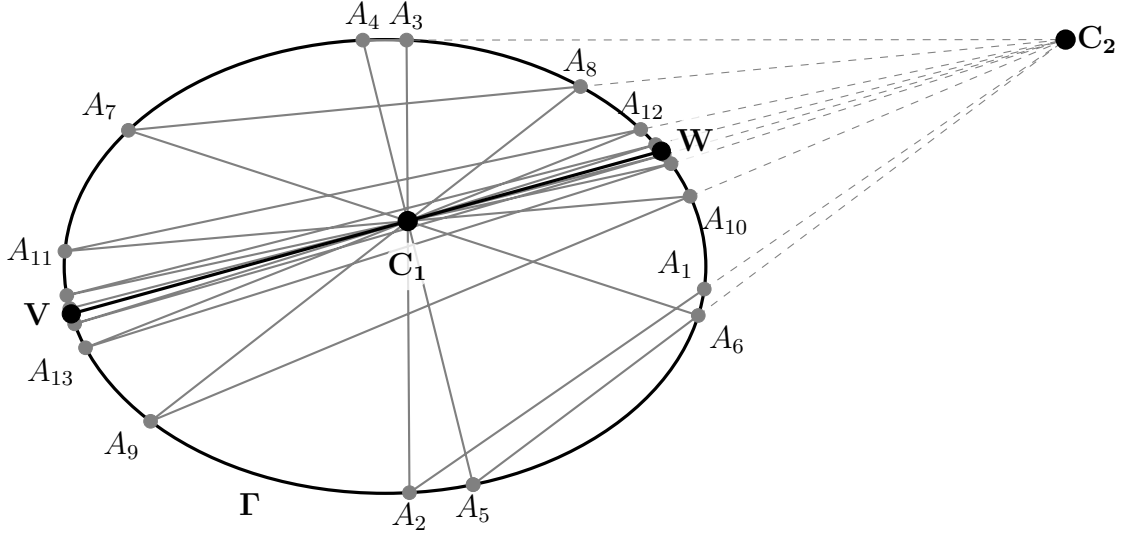
\begin{figure}[ht]
	\begin{center}
		\begin{tikzpicture}[scale=3]
			\draw[very thick,black](0,0) ellipse (1.41421 and 1);		
			\draw[gray,thick] (1.40715, -0.0998334)--(0.107671, -0.997098)--(0.094888, 0.997747)--(-0.0979835, 0.997597)--(0.388329, -0.961561)--(1.38078, -0.216156)--(-1.13129, 0.600076)--(0.861582, 0.792993)--(-1.03351, -0.682589)--(1.34477, 0.309516)--(-1.41103, 0.0670584)--(1.12721, 0.603902)--(-1.32024, -0.358442)--(1.26071, 0.453104)--(-1.40264, -0.127665)--(1.19288, 0.537136)--(-1.36823, -0.252925)--(1.23078, 0.492539)--(-1.3897, -0.185395);
			
			\draw[gray,dashed] (1.40715, -0.0998334)--(3,1)--(0.094888, 0.997747);
			\draw[gray,dashed](1.38078, -0.216156)--(3,1)--(0.861582, 0.792993);
			\draw[gray,dashed](1.34477, 0.309516)--(3,1)--(1.12721, 0.603902);
			\draw[gray,dashed](1.26071, 0.453104)--(3,1)--(1.19288, 0.537136);
			\draw[gray,dashed](1.23078, 0.492539)--(3,1);

			\draw[gray,fill=gray](1.40715, -0.0998334) circle (0.03) node[above left,black]{$A_1$};
			\draw[gray,fill=gray](0.107671, -0.997098) circle (0.03) node[below,black]{$A_2$};
			\draw[gray,fill=gray](0.094888, 0.997747) circle (0.03)node[above,black]{$A_3$};
			\draw[gray,fill=gray](-0.0979835, 0.997597) circle (0.03)node[above,black]{$A_4$};
			\draw[gray,fill=gray](0.388329, -0.961561) circle (0.03)node[below,black]{$A_5$};
			\draw[gray,fill=gray](1.38078, -0.216156) circle (0.03)node[below right,black]{$A_6$};
			
			\draw[gray,fill=gray](-1.13129, 0.600076) circle (0.03)node[above 
			left,black]{$A_7$};
			\draw[gray,fill=gray](0.861582, 0.792993) circle (0.03)node[above,black]{$A_8$};
			\draw[gray,fill=gray](-1.03351, -0.682589) circle (0.03)node[below left,black]{$A_9$};
			\draw[gray,fill=gray](1.34477, 0.309516) circle (0.03)node[below right,black]{$A_{10}$};
			\draw[gray,fill=gray](-1.41103, 0.0670584) circle (0.03)node[left,black]{$A_{11}$};
			\draw[gray,fill=gray](1.12721, 0.603902) circle (0.03)node[above,black]{$A_{12}$};
			
			\draw[gray,fill=gray](-1.32024, -0.358442) circle (0.03)node[below left,black]{$A_{13}$};
			\draw[gray,fill=gray](1.26071, 0.453104) circle (0.03);
			\draw[gray,fill=gray](-1.40264, -0.127665) circle (0.03);
			\draw[gray,fill=gray](1.19288, 0.537136) circle (0.03);
			\draw[gray,fill=gray](-1.36823, -0.252925) circle (0.03);
			\draw[gray,fill=gray](1.23078, 0.492539) circle (0.03);
			\draw[gray,fill=gray](-1.3897, -0.185395) circle (0.03);
			\draw[gray,fill=gray](0.388329, -0.961561) circle (0.03);

			\draw[black,thick,fill=black](3, 1) circle (0.04) node[right]{$\mathbf{C_2}$};
			\draw[black,thick,fill=black](0.1,0.2) circle (0.04);
			
			\draw[white,thick,fill=white,opacity=0.8](0.1,0) circle (0.1);
			\node at (0.1,0)  {$\mathbf{C_1}$};

	\draw[black,very thick](1.21784,0.508368)--(-1.38295, -0.209091);
\draw[black,fill=black](1.21784,0.508368) circle (0.04);
\draw[white,thick,fill=white,opacity=0.8](1.36784,0.558368) circle (0.1);
\node at (1.36784,0.558368)  {$\mathbf{W}$};

\draw[black,fill=black](-1.38295, -0.209091) circle (0.04); 
\node at (-1.53295, -0.209091) {$\mathbf{V}$};

		\node[below left] at (-0.5,-0.95) { $\mathbf{\Gamma}$};

		\end{tikzpicture}
		\caption{A polygonal line inscribed in a smooth conic $\Gamma$ and circumscribed about a singular conic which is given by a pair of points $C_1$, $C_2$. In this figure, point $C_1$ is inside the conic, and $C_2$ is outside.}\label{fig:asymp}
	\end{center}
\end{figure}

As a consequence of Proposition \ref{prop:asymp}, we can conclude that closed polygonal lines might appear only if the line $C_1C_2$ is disjoint with $\Gamma$.
Such an example is shown in Figure \ref{fig:C-singular}.

\begin{example}\label{ex:light-like}
Light-like billiard trajectories within an ellipse in the Lorentzian plane represent an example of polygonal lines inscribed in a smooth conic and circumscribed about a singular one, which is determined by the null-directions of the plane.
See \cite{DR2012adv} for detailed discussion and the analytic conditions for periodicity.
\end{example}

Generalising the discussion from \cite{DR2012adv}, we will now derive the conditions for closed polygons.
First, we need to find an appropriate coordinate system.

\begin{lemma}\label{lem:coordinate-system}
Suppose that $\Gamma$ is a smooth conic, and $\mathcal{C}^*$ a singular one given by a pair of points $C_1$ and $C_2$.
If the line $C_1C_2$ is disjoint with $\Gamma$, then we can choose a coordinate system such that $C_1C_2$ is the line at the infinity and $\Gamma$ is a circle.

Moreover, it is possible to choose such a coordinate system so that points $C_1$, $C_2$ have homogeneous coordinates $[1:0:0]$ and $[\alpha:1:0]$, with $\alpha\ge0$.
\end{lemma}
\begin{proof}
We place the coordinate system in the projective plane such that $C_1C_2$ is the line at the infinity.
Since the smooth conic $\Gamma$ is disjoint with that line, it follows that $\Gamma$ is an ellipse.
Now, with an appropriate affine transformation, we can map it to a circle.
Finally, an appropriate rotation or a reflection of the affine plane will place $C_1$, $C_2$ into points with the listed coordinates.
\end{proof}

Using Lemma \ref{lem:coordinate-system} and similarly as in \cite{DR2012adv}, we can now derive the closure condition.

\begin{proposition}\label{prop:cayley-sing}
Suppose that $\Gamma$ is a smooth conic and $\mathcal{C}^*$ a singular one given by a pair of points $C_1$ and $C_2$, such that the line $C_1C_2$ is disjoint with $\Gamma$.
Assume that the coordinate system is chosen as in Lemma \ref{lem:coordinate-system}.

Then the polygonal lines inscribed in $\Gamma$ and circumscribed about $\mathcal{C}^*$ satisfy:
\begin{itemize}
\item there are no such closed polygons with odd number of sides;
\item they are closed with $2n$ sides if and only if
$$
\arctan\left(\frac1{\alpha}\right)
\in
\left\{
\frac{k\pi}n \mid 1\le k< n, (k,n)=1
\right\}.
$$
\end{itemize}
\end{proposition}
\begin{proof}
The first statement follows from the fact that the sides of any polygon circumscribed about $\mathcal{C}^*$ alternately contain points $C_1$ and $C_2$.

The angle between directions corresponding to points $C_1$ and $C_2$ is $\arctan(1/\alpha)$, thus the consecutive even vertices of the polygonal line correspond to a rotation by that angle, which immediately implies the second statement.
\end{proof}

\subsection{Circumscribed conic is singular}

Now we will suppose that $\Gamma$ is singular, while $\mathcal{C}$ is smooth.
In this case, according to the discussion from the very beginning of Section \ref{sec:singular}, conic $\Gamma$ will be the union of two distinct line in the plane. 
Notice that this situation is dual to the case considered is Section \ref{sec:in-singular}, thus all discussion from there can be carried on in this case as well.

\begin{proposition}[Dual of Proposition \ref{prop:asymp}]
Let $\mathcal{C}$ be a smooth conic in the real plane, and $\Gamma=g_1\cup g_2$ a singular one, given by a pair of lines $g_1$, $g_2$.
Suppose that the intersection point of $g_1$ and $g_2$ does not lie within $\mathcal{C}$.
	
Then all the polygonal lines inscribed in $\Gamma$ and circumscribed about $\mathcal{C}$ satisfy the following:
	\begin{itemize}
		\item there is one such degenerate polygonal line consisting of pair of tangent lines to $\mathcal{C}$ from the intersection point $g_1\cap g_2$;
		\item all other polygonal lines have infinitely many segments, and their vertices asymptotically tend to $g_1\cap g_2$.
		For an illustration, see Figure \ref{fig:asym}.
	\end{itemize}
\end{proposition}
 \begin{figure}[ht]
	\begin{center}
		\begin{tikzpicture}[scale=1.8]

			\clip (-2,-2) rectangle (8,3);			
			
			\begin{scope}[rotate=90]

				\draw[thick,gray](-0.64444, 0.933321)--(-9.87346, -0.0632697)--(0.000777501, -1.00233)--(-6.68248, -1.65876)--(-0.599314, 0.797941)--(-1.39757, -4.30122)--(0.136422, -1.40927)--(6.3751, -8.18755)--(-1.65894, 3.97681)--(0.460146, -5.23007)--(0.578674, -2.73602)--(2.44711, -6.22356)--(5.54655, -17.6397)--(1.22372, -5.61186)--(1.10594, -4.31781)--(1.82371, -5.91185)--(2.05581, -7.16742)--(1.48572, -5.74286)--(1.42217, -5.2665)--(1.66345, -5.83172)--(1.71277, -6.13832)--(1.56631, -5.78315)--(1.54443, -5.63329)--(1.61833, -5.80916)--(1.63144, -5.89433)--(1.59016, -5.79508)--(1.58347, -5.75042)--(1.60532, -5.80266)--(1.60904, -5.82712)--(1.59713, -5.79857)--(1.59516, -5.78548)--(1.60155, -5.80077)--(1.60262, -5.80786)--(1.59917, -5.79958)--(1.59859, -5.79577)--(1.60045, -5.80023)--(1.60076, -5.80229)--(1.59976, -5.79988);

				\draw[very thick,black](0,0) ellipse (0.707107 and 1);
				
				\draw[thick,black,dashed](-9.87346, -0.0632697)--(6.3751, -8.18755);
				
				\draw[thick,black,dashed](-1.65894, 3.97681)--(5.54655, -17.6397);
				
				\draw[black,thick,fill=gray](-0.64444, 0.933321) circle (0.035); \draw[black,thick,fill=gray](-9.87346, -0.0632697) circle (0.035); \draw[black,thick,fill=gray](0.000777501, -1.00233) circle (0.035); \draw[black,thick,fill=gray](-6.68248, -1.65876) circle (0.035); \draw[black,thick,fill=gray](-0.599314, 0.797941) circle (0.035); \draw[black,thick,fill=gray](-1.39757, -4.30122) circle (0.035); \draw[black,thick,fill=gray](0.136422, -1.40927) circle (0.035); \draw[black,thick,fill=gray](6.3751, -8.18755) circle (0.035); \draw[black,thick,fill=gray](-1.65894, 3.97681) circle (0.035); \draw[black,thick,fill=gray](0.460146, -5.23007) circle (0.035); \draw[black,thick,fill=gray](0.578674, -2.73602) circle (0.035); \draw[black,thick,fill=gray](2.44711, -6.22356) circle (0.035); \draw[black,thick,fill=gray](5.54655, -17.6397) circle (0.035); \draw[black,thick,fill=gray](1.22372, -5.61186) circle (0.035); \draw[black,thick,fill=gray](1.10594, -4.31781) circle (0.035); \draw[black,thick,fill=gray](1.82371, -5.91185) circle (0.035); \draw[black,thick,fill=gray](2.05581, -7.16742) circle (0.035); \draw[black,thick,fill=gray](1.48572, -5.74286) circle (0.035); \draw[black,thick,fill=gray](1.42217, -5.2665) circle (0.035); \draw[black,thick,fill=gray](1.66345, -5.83172) circle (0.035); \draw[black,thick,fill=gray](1.71277, -6.13832) circle (0.035); \draw[black,thick,fill=gray](1.56631, -5.78315) circle (0.035); \draw[black,thick,fill=gray](1.54443, -5.63329) circle (0.035); \draw[black,thick,fill=gray](1.61833, -5.80916) circle (0.035); \draw[black,thick,fill=gray](1.63144, -5.89433) circle (0.035); \draw[black,thick,fill=gray](1.59016, -5.79508) circle (0.035); \draw[black,thick,fill=gray](1.58347, -5.75042) circle (0.035); \draw[black,thick,fill=gray](1.60532, -5.80266) circle (0.035); \draw[black,thick,fill=gray](1.60904, -5.82712) circle (0.035); \draw[black,thick,fill=gray](1.59713, -5.79857) circle (0.035); \draw[black,thick,fill=gray](1.59516, -5.78548) circle (0.035); \draw[black,thick,fill=gray](1.60155, -5.80077) circle (0.035); \draw[black,thick,fill=gray](1.60262, -5.80786) circle (0.035); \draw[black,thick,fill=gray](1.59917, -5.79958) circle (0.035); \draw[black,thick,fill=gray](1.59859, -5.79577) circle (0.035); \draw[black,thick,fill=gray](1.60045, -5.80023) circle (0.035); \draw[black,thick,fill=gray](1.60076, -5.80229) circle (0.035); \draw[black,thick,fill=gray](1.59976, -5.79988) circle (0.035);	
	
			\node[below left] at (-0.1,1) { $\mathcal{C}$};
			
			\node at (2.8,-6.2) { $g_1$};
			\node at (2.3,-7.5) { $g_2$};

			\end{scope}

		\end{tikzpicture}
		\caption{Polygon inscribed in a singular conic, represented by a pair dashed lines $g_1$ and $g_2$, and circumscribed about a smooth conic $\mathcal{C}$, represented as a solid ellipse. Since the point of intersection of $g_1$ and $g_2$ is outside $\mathcal{C}$, the vertices of the polygon will approach that intersection.
			The only exception is a degenerate polygon consisting only of two lines, both tangent to $\mathcal{C}$ and containing the point $g_1\cap g_2$.
		}\label{fig:asym}
	\end{center}
\end{figure}
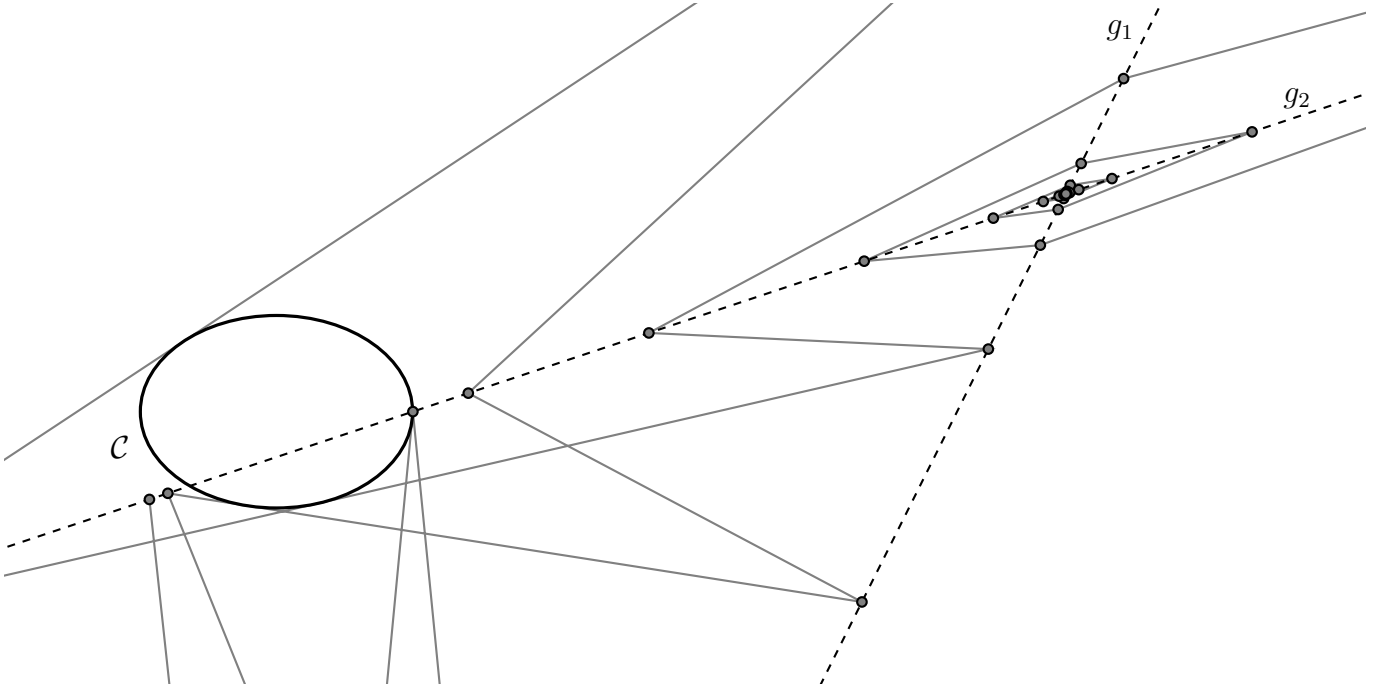

\begin{proposition}[Dual of Proposition \ref{prop:cayley-sing}]
In the real plane, suppose that $\mathcal{C}$ is given a smooth conic, while $\Gamma$ is a singular conic consisting of two distinct lines $g_1$, $g_2$, such that the intersection point of $g_1$ and $g_2$ lies within $\mathcal{C}$.
Also, suppose that a coordinate system is conveniently chosen so that $\mathcal{C}$ is a circle and $g_1$, $g_2$ have equations $x=0$ and $\alpha x+y=0$, with $\alpha\ge0$.

Then the polygonal lines inscribed in $\Gamma$ and circumscribed about $\mathcal{C}$ satisfy:
\begin{itemize}
	\item there are no such closed polygons with odd number of sides;
	\item they are closed with $2n$ sides if and only if
	$$
	\arctan\left(\frac1{\alpha}\right)
	\in
	\left\{
	\frac{2k\pi}n \mid 1\le k< n, (k,n)=1
	\right\}.
	$$
\end{itemize}

\end{proposition}

\begin{example}
Hexagons inscribed in a singular conic $\Gamma$ and circumscribed about a smooth conic $\mathcal{C}$ are shown in Figure \ref{fig:Gama-singular}.	
\end{example}
 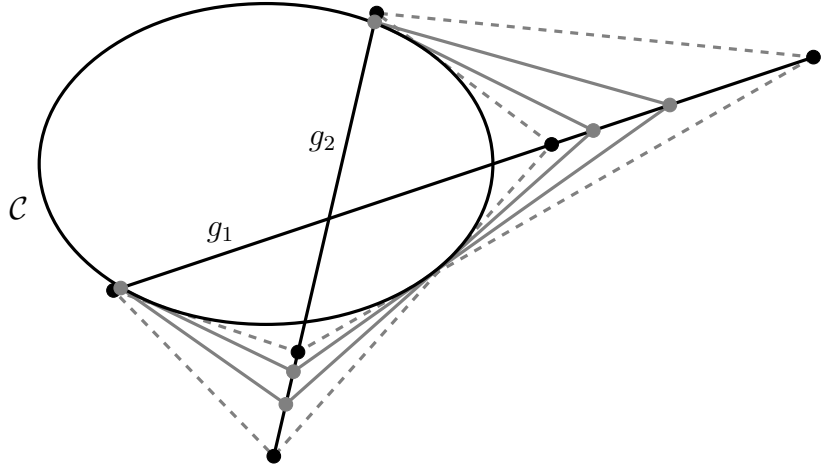
\begin{figure}[ht]
	\begin{center}
		\begin{tikzpicture}[scale=3]

\begin{scope}[rotate=90]

			\draw[very thick,dashed,gray](0.471027, -2.41308)--(-0.829036, -0.14078)--(-0.557454, 0.672361)--(-1.28807, -0.0341635)--(0.0862805, -1.25884)--(0.665086, -0.487807)--cycle;
			
%\draw[very thick,red](1.25396, -4.76189)--(-0.763439, -0.156016)--(-0.586165, 0.758494)--(-1.84598, 0.0954161)--(0.0424339, -1.1273)--(0.774032, -0.513111)--cycle;

\draw[very thick,gray](-0.546852, 0.640557)--(-0.915369, -0.120729)--(0.260003, -1.78001)--(0.626298, -0.478798)--(0.147582, -1.44275)--(-1.05926, -0.0873085)--cycle;
			
\draw[very thick,black](0.471027, -2.41308)--(-0.557454, 0.672361);
\draw[very thick,black](0.665086, -0.487807)--(-1.28807, -0.0341635);	

			\draw[very thick,black](0,0) ellipse (0.707107 and 1);

\draw[black,fill=black] (0.471027, -2.41308) circle (0.03); \draw[black,fill=black](-0.829036, -0.14078) circle (0.03); \draw[black,fill=black](-0.557454, 0.672361) circle (0.03); \draw[black,fill=black](-1.28807, -0.0341635) circle (0.03); \draw[black,fill=black](0.0862805, -1.25884) circle (0.03); \draw[black,fill=black](0.665086, -0.487807) circle (0.03);

\draw[gray,fill=gray](-0.546852, 0.640557) circle (0.03); \draw[gray,fill=gray](-0.915369, -0.120729) circle (0.03); \draw[gray,fill=gray](0.260003, -1.78001) circle (0.03); \draw[gray,fill=gray](0.626298, -0.478798) circle (0.03); \draw[gray,fill=gray](0.147582, -1.44275) circle (0.03); \draw[gray,fill=gray](-1.05926, -0.0873085) circle (0.03);

\node[below left] at (-0.1,1) { $\mathcal{C}$};

\node at (-0.3,.2) { $g_1$};
\node at (0.1,-0.25) { $g_2$};

\end{scope}
			
		\end{tikzpicture}
		\caption{Hexagons inscribed in a singular conic consisting of lines $g_1$, $g_2$ and circumscribed about a smooth conic $\mathcal{C}$.}\label{fig:Gama-singular}
	\end{center}
\end{figure}

\subsection{Both conics are singular}

In this section we consider polygonal lines inscribed in a singular conic and circumscribed about another one.
The circumscribed conic $\Gamma$ is then a pair of lines $g_1$, $g_2$, while the inscribed conic $\mathcal{C}$ is represented by a pair of points $C_1$, $C_2$.

\begin{example}\label{ex:harmonic}
It is well known that four collinear points $(X, Y, Z, W)$ are harmonically conjugate if and only if there is a quadrangle $ABCD$ such that $X=AB\cap CD$, $Y=AD\cap BC$, $Z=BD\cap XY$, $W=AC\cap XY$. 
The construction does not depend on the choice of the quadrangle and this follows from the Desargues theorem, see e.g.~\cite{Wylie}.
Using this construction of harmonic conjugacy via quadrangles, one can see that, if the lines $g_1$, $g_2$ intersect $C_1C_2$ in a pair of points $(D_1, D_2)$, which is harmonically conjugated with the pair $C_1$, $C_2$, then there will be a family of infinitely many quadrangles inscribed in $\Gamma$ and circumscribed about $\mathcal{C}$, as illustrated in Figure \ref{fig:harmonic}. Let us recall also that the cross-ratio of the harmonically-conjugated points is equal to $-1$. %\sout{Moreover, the same will be true for any other singular conic $\Gamma'$ which consists of a pair of lines $g_1', g_2'$, such that $g_1'$ contains $D_1$ and $g_2'$ contains $D_2$.}
\end{example}
\begin{figure}[ht]
	\begin{center}
		\begin{tikzpicture}[scale=1]
			
\begin{scope}[rotate=90]
\draw[gray,thick] (1, 9)--(-1.14286, 2.57143)--(0.117647, 1.05882)--(1., 1.5)--cycle;

\draw[gray,thick] (0.3, 2.7)--(-0.207792, 2.1039)--(0.162162, 1.45946)--(0.553191, 1.7234)--cycle;

\draw[gray,thick] (-0.5, -4.5)--(-3.2, 3.6)--(0.0869565, 0.782609)--(1.42857, 1.28571)--cycle;

\draw[very thick,black](-3.2,3.6)--(4,0);%(1.42857, 1.28571);
\draw[very thick,black](1,9)--(-0.5,-4.5);

\draw[black,fill=black](-2,0) circle (0.08) node[below right]{$C_2$};
\draw[black,fill=black](1,0) circle (0.08) node[above right]{$C_1$};

\draw[black,fill=black](4,0) circle (0.08) node[above right]{$D_1$};
\draw[black,fill=black](0,0) circle (0.08) node[above right]{$D_2$};

\draw[black,dashed](-2,0)--(4,0);

\draw[gray,fill=gray](1, 9) circle (0.08); 
\draw[gray,fill=gray](-1.14286, 2.57143) circle (0.08); \draw[gray,fill=gray](0.117647, 1.05882) circle (0.08); \draw[gray,fill=gray](1., 1.5) circle (0.08);

\draw[gray,fill=gray](-0.5, -4.5) circle (0.08); 
\draw[gray,fill=gray](-3.2, 3.6) circle (0.08); 
\draw[gray,fill=gray](0.0869565, 0.782609) circle (0.08); \draw[gray,fill=gray](1.42857, 1.28571) circle (0.08);

\draw[gray,fill=gray](0.3, 2.7) circle (0.08); 
\draw[gray,fill=gray](-0.207792, 2.1039) circle (0.08); \draw[gray,fill=gray](0.162162, 1.45946) circle (0.08); \draw[gray,fill=gray](0.553191, 1.7234) circle (0.08);

\node at (3,1) {$g_1$};
\node at (0.15,3.8) {$g_2$};
\end{scope}

		\end{tikzpicture}
\caption{The circumscribed conic $\Gamma$ is singular and is represented by a pair of lines $g_1$, $g_2$, which are black in this figure. The singular inscribed conic $\mathcal{C}^*$ is represented by a pair of points $C_1$, $C_2$.
The lines $g_1$, $g_2$	intersect $C_1C_2$ at $D_1$, $D_2$.
In this case, the quadruple of points $(C_1,C_2,D_1,D_2)$ is harmonically conjugate, thus there are quadrangles inscribed in $\Gamma$ and circumscribed about $\mathcal{C}^*$, see Example \ref{ex:harmonic}.
The	quadrangles are colored gray in this figure.
}\label{fig:harmonic}
	\end{center}
\end{figure}
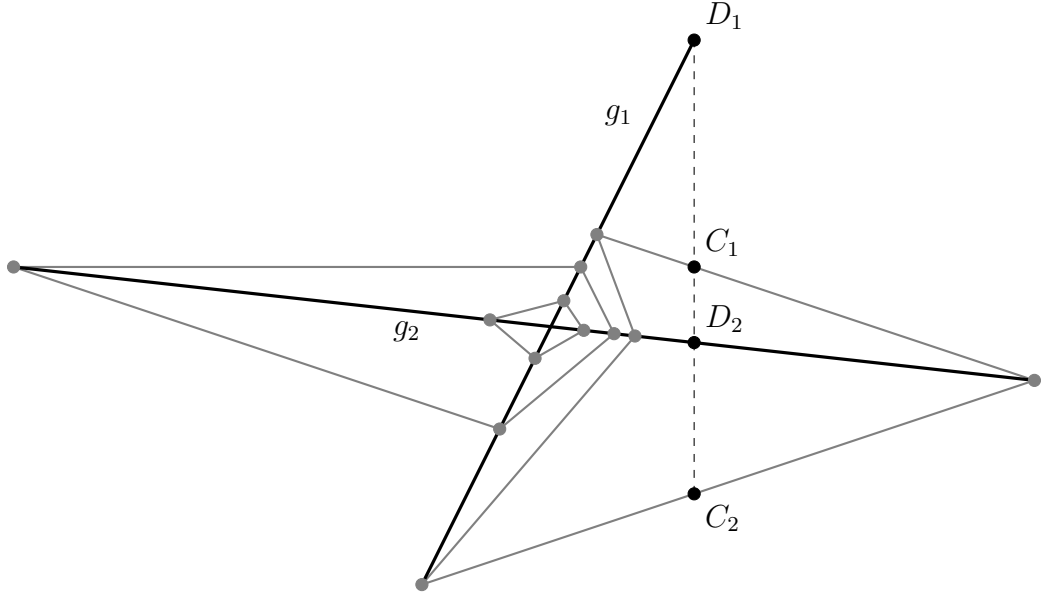

In order to discuss a more general case, we will introduce a convenient coordinate system in the real projective plane.

\begin{lemma}\label{lem:2sing-coord}
For a given pair of distinct lines $g_1$, $g_2$ and a given pair of distinct points $C_1$, $C_2$, one can find a coordinate system such that:
\begin{itemize}
	\item the lines $g_1$, $g_2$ intersect each other on the line at the infinity, i.e.~they are parallel in the affine part of the plane;
	\item point $C_2$ is also on the line at the infinity.
\end{itemize}
\end{lemma}

\begin{example}
After a coordinate change as in Lemma \ref{lem:2sing-coord}, in the configuration from Example \ref{ex:harmonic}, the distances of the point $C_1$ from $g_1$, $g_2$ are equal, see Figure \ref{fig:equal}.
Let this distance from $C_1$ to $g_1$ be equal to $t$. We can denote by $\Gamma_t$ the degenerate conic which is the union of two parallel lines in a given direction, at the distance $t$ from $C_1$. Then, $\Gamma_t$ for varying $t$ and $\mathcal C^*$ form an isoperiodic pencil of conics, similar to those studied in \cite{DR2024geom}. In \cite{DR2024geom}, isoperiodic families of conics were used to generate explicite algebraic solutions to Painlev\'e VI equations, following \cites{DS2019,Hitchin,Okamoto1987,Picard}. Since here the cross-ratio of the points $C_1, C_2, D_1, D_2$ is fixed, as it is equal to $-1$ for the harmonically conjugated points, the isoperiodic family will not produce nontrivial solutions to Painlev\'e VI equations in this case.
\end{example}
\begin{figure}[ht]
	\begin{center}
		\begin{tikzpicture}[scale=1]

\draw[gray,thick] (0,1)--(0,-1)--(2.5, 1)--(-2.5, -1)--cycle;

\draw[gray,thick] (-1,1)--(1,-1)--(3.5, 1)--(-3.5, -1)--cycle;
				
				\draw[very thick,black](-4,1)--(4,1) node[right]{$g_1$};
				\draw[very thick,black](-4,-1)--(4,-1) node[right]{$g_2$};
				
				\draw[black,fill=black](0,0) circle (0.08);
				\node at (-0.3,-0.5){$C_1$};

\draw[gray,fill=gray](0,1) circle (0.08); 
\draw[gray,fill=gray](0,-1) circle (0.08); 
\draw[gray,fill=gray](2.5, 1) circle (0.08); 
\draw[gray,fill=gray](-2.5, -1) circle (0.08); 
\draw[gray,fill=gray](-1,1) circle (0.08); 
\draw[gray,fill=gray](1,-1) circle (0.08); 
\draw[gray,fill=gray](3.5, 1) circle (0.08); 
\draw[gray,fill=gray](-3.5, -1) circle (0.08);

		\end{tikzpicture}
		\caption{The circumscribed conic $\Gamma$ is singular and is represented by a pair of parallel lines $g_1$, $g_2$, which are black in this figure. The singular inscribed conic $\mathcal{C}^*$ is represented by a pair of points $C_1$, $C_2$, where $C_1$ is equally distanced from $g_1$ and $g_1$, while $C_2$ lies on the line at infinity and thus is not visible in this figure.
	The	quadrangles inscribed in $\Gamma$ and circumscribed about $\mathcal{C}^*$ are colored gray.
		}\label{fig:equal}
	\end{center}
\end{figure}
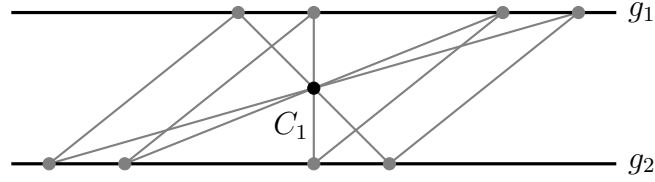

On the other hand, we have the following:

\begin{lemma}
Suppose that, in a coordinate system which is chosen as in Lemma \ref{lem:2sing-coord}, the distances of $C_1$ to $g_1$, $g_2$ are not equal.
Then the polygonal lines inscribed in $g_1\cup g_2$ and circumscribed about $C_1$, $C_2$ satisfy the following:
\begin{itemize}
	\item the only finite such polygonal line consists of a single segment, determined by the intersections of $g_1$ and $g_2$ with the line $C_1C_2$;
	\item all other such polygonal lines are not closed. Moreover, their sides, in one direction, approach the line $C_1C_2$, and in the other direction they diverge to infinity.
\end{itemize}
\end{lemma}
\begin{proof}
One can easily show this using Figure \ref{fig:not-center}.
\end{proof}

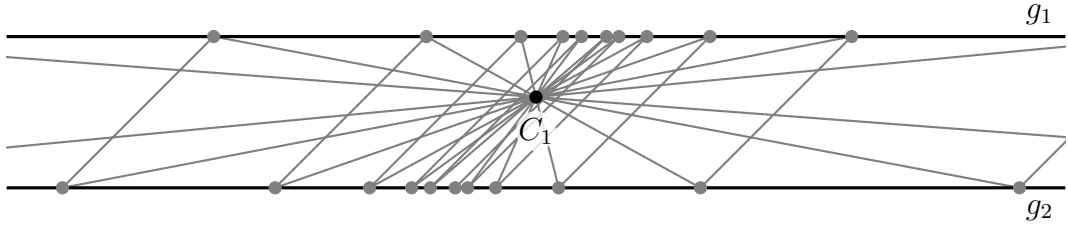
\begin{figure}[ht]
	\begin{center}
		\begin{tikzpicture}[scale=1]
			\clip (-7,-1.5) rectangle (7,1.5);

\draw[gray,thick](15.8859,-1)--(-10.5906,1)--(-12.5906,-1)--(8.39375,1)--(6.39375,-1)--(-4.2625,1)--(-6.2625,-1)--(4.175,1)--(2.175,-1)--(-1.45,1)--(-3.45,-1)--(2.3,1)--(0.3,-1)--(-0.2,1)--(-2.2,-1)--(1.46667,1)--(-0.533333,-1)--(0.355556,1)--(-1.64444,-1)--(1.0963,1)--(-0.903704,-1)--(0.602469,1)--(-1.39753,-1)--(0.931687,1)--(-1.06831,-1);
			
			\draw[very thick,black](-7,1)--(7,1) node[above left]{$g_1$};
			\draw[very thick,black](-7,-1)--(7,-1) node[below left]{$g_2$};
			
\draw[white,thick,fill=white,opacity=0.8](0,-0.25) circle (0.25);
\node at (0,-0.25)  {$C_1$};

			\draw[black,fill=black](0,0.2) circle (0.08);

\draw[gray,fill=gray](15.8859,-1) circle (0.08); 
\draw[gray,fill=gray](-10.5906,1) circle (0.08); \draw[gray,fill=gray](-12.5906,-1) circle (0.08); \draw[gray,fill=gray](8.39375,1) circle (0.08); \draw[gray,fill=gray](6.39375,-1) circle (0.08); \draw[gray,fill=gray](-4.2625,1) circle (0.08); \draw[gray,fill=gray](-6.2625,-1) circle (0.08); 
\draw[gray,fill=gray](4.175,1) circle (0.08); 
\draw[gray,fill=gray](2.175,-1) circle (0.08); 
\draw[gray,fill=gray](-1.45,1) circle (0.08); 
\draw[gray,fill=gray](-3.45,-1) circle (0.08); 
\draw[gray,fill=gray](2.3,1) circle (0.08); 
\draw[gray,fill=gray](0.3,-1) circle (0.08); 
\draw[gray,fill=gray](-0.2,1) circle (0.08); 
\draw[gray,fill=gray](-2.2,-1) circle (0.08); 
\draw[gray,fill=gray](1.46667,1) circle (0.08); \draw[gray,fill=gray](-0.533333,-1) circle (0.08); \draw[gray,fill=gray](0.355556,1) circle (0.08); \draw[gray,fill=gray](-1.64444,-1) circle (0.08); \draw[gray,fill=gray](1.0963,1) circle (0.08); \draw[gray,fill=gray](-0.903704,-1) circle (0.08); \draw[gray,fill=gray](0.602469,1) circle (0.08); \draw[gray,fill=gray](-1.39753,-1) circle (0.08); \draw[gray,fill=gray](0.931687,1) circle (0.08); \draw[gray,fill=gray](-1.06831,-1) circle (0.08);

		\end{tikzpicture}
		\caption{The circumscribed conic $\Gamma$ is singular and is represented by a pair of parallel lines $g_1$, $g_2$, which are black in this figure. The singular inscribed conic $\mathcal{C}^*$ is represented by a pair of points $C_1$, $C_2$, where $C_1$ is not equally distanced from $g_1$ and $g_1$, while $C_2$ lies on the line at infinity and thus is not visible in this figure.
		An infinite polygonal line inscribed in $\Gamma$ and circumscribed about $\mathcal{C}^*$ is colored gray.
		}\label{fig:not-center}	
	\end{center}
\end{figure}

From the above considerations, we get our final result:

\begin{theorem} 
Let two degenerate conics be given in the real plane: $\Gamma$ given as a pair of lines $g_1$, $g_2$; and $\mathcal{C}^*$ represented by a pair of points $C_1$, $C_2$.  Consider polygonal lines inscribed in $\Gamma$ and circumscribed about $\mathcal{C}^*$. The polygonal line is periodic with the period $n$ if and only if $n=4$ and the lines $g_1$ and $g_2$ intersect the line $C_1C_2$ in a pair of points which is harmonically conjugated with the pair $C_1$, $C_2$.
\end{theorem}

\subsection*{Acknowledgment}
 The authors are grateful to Valery Vasilievich Kozlov for the opportunity to learn a lot about various fields of mathematics and mechanics from his books, papers, and lectures, and wish him many happy returns. The authors thank the referees for their helpful comments and suggestions.
This research  was supported
by the Australian Research Council, Discovery Project 190101838 \emph{Billiards within quadrics and beyond}, the Serbian Ministry of Science, Technological Development and Innovation and the Science Fund of Serbia grant IntegraRS, and the Simons Foundation grant no.~854861.

\end{document}